\documentclass[conference]{IEEEtran}

\usepackage{times,amsmath,epsfig}
\usepackage{booktabs}
\usepackage{breqn}
\usepackage[compress]{cite}
\usepackage{tabularx}
\usepackage[font={small}]{caption}
\usepackage{amsfonts}
\usepackage{algorithm}
\usepackage{algorithmic}
\usepackage{latexsym}
\usepackage{tikz}
\usetikzlibrary{calc}
\usepackage{amssymb}
\usepackage{bm}
\usepackage{flexisym}
\usepackage{flushend}
\usepackage{dsfont}
\usetikzlibrary{positioning}
\usetikzlibrary{arrows}
\usetikzlibrary{chains,shapes.multipart}
\usetikzlibrary{shapes}
\usetikzlibrary{automata,}
\usepackage{multirow}
\usepackage{latexsym}
\definecolor {processblue}{cmyk}{0,0,0,0.17}
\definecolor{light-gray}{gray}{0.85}
\usepackage{comment}
\usepackage{amsthm}
 \usepackage[usenames,dvipsnames]{pstricks}

\newtheorem{thm}{Theorem}
\newtheorem{lem}[thm]{Lemma}

\newcommand{\bphi}{{\boldsymbol \phi}}
\newcommand{\bzero}{{\boldsymbol 0}}

\begin{document}
\title{\huge Energy Storage Sizing for Peak Hour Utility Applications}

\author{\IEEEauthorblockN{I. Safak Bayram\IEEEauthorrefmark{1},
Ali Tajer\IEEEauthorrefmark{2},
Mohamed Abdallah\IEEEauthorrefmark{1},
and Khalid Qaraqe\IEEEauthorrefmark{1}
\IEEEauthorblockA{\IEEEauthorrefmark{1}Department of Electrical and Computer Engineering, Texas A\&M University at Qatar, Doha, Qatar}\IEEEauthorblockA{\IEEEauthorrefmark{2}Department of Electrical, Computer, and Systems Engineering, Rensselaer Polytechnic Institute, Troy, NY, USA}
Emails: \{islam.bayram, mohamed.abdallah, khalid.qaraqe\}@qatar.tamu.edu, tajer@ecse.rpi.edu 
}

}

\maketitle


\begin{abstract}
In future smart grids, energy storage systems (ESSs) are expected to play a key role in reducing peak hour electricity generation cost and the associated level of carbon emissions. Considering their high acquisition, operation, and maintenance costs, ESSs are likely to serve a large number of users. Hence, optimal sizing of energy ESSs plays a critical role as over-provisioning ESS size leads to under-utilizing costly assets and under-provisioning it taxes operation lifetime. This paper proposes a stochastic framework for analyzing the optimal size of energy storage systems. In this framework the demand of each customer is modeled stochastically and the aggregate demand is accommodated by a combination of power drawn from the grid and the storage unit when the demand exceed grid capacity. In this framework an analytical method is developed, which provides tractable solution to the ESS sizing problem of interest. The results indicate that significant savings in terms of ESS size can be achieved.

\end{abstract} 

\IEEEpeerreviewmaketitle

\section{Introduction}
There is a growing need for reducing the use of hydrocarbons and the cost of electricity during electricity consumption peak hours. One effective way to achieve this, is deploying energy storage systems (ESSs) which can store lower cost energy, through either renewables or off-peak hour grid power, and discharge the stored energy into the grid during peak load periods. Furthermore, storage units can improve power system reliability by supplying standby power during outages and reduce the load on the equipments, thereby decreasing the aging pace of network components.  As other benefits, energy storage can foster the adoption of intermittent distributed energy generation into the distribution network, aid grid operations by improving power quality (e.g., mitigating voltage sags and flickers) and efficiency, regulating frequency, and enabling active customer involvement in demand response programs. Some of the benefits of deploying energy storage systems are summarized in Table \ref{ESSBenefits}.

While deploying storage units has certain benefits, their deployment based on the existing ESS technologies \cite{sandiaLabs} is costly. Therefore, optimal sizing of the storage units based on the realistic needs of the grids is a critical step for efficient operation of the grid. Specifically, over-provisioning ESS size entails costly and underutilized assets, whereas under-provisioning reduces its operating lifetime (e.g., frequently exceeding allowable depth of charge level degrades its health). Hence, there is a strong need to develop analytical models to solve the sizing problem.

ESS sizing has received some attention in the literature. The work in~\cite{oudalov2007sizing} presents a sizing approach for {\em single} industrial customers for peak saving applications. The sizing problem is solved through maximizing the net benefits, which is the sum of reductions in the electricity bills minus the operation costs and one time acquisition cost. Similarly, \cite{sizingNAN} proposes a sizing framework using similar cost models for a micro grid, but it also consider savings due to storage of energy generated by renewable resources. From the power engineering point of view, the sizing problem is usually solved via simulation techniques \cite{lukicESS} and for wind farm applications, ESS is used to convert such intermittent and non-dispatchable sources into dispatchable ones \cite{sizing1,sizing2}. However, simulations techniques are usually computationally expensive and the success of the proper sizing requires availability of data traces.

\begin{table}[t]
\caption{Benefits of Energy Storage Systems by Users \cite{eckroad2003epri}}\label{ESSBenefits}
\vspace{-0pt}
\begin{tabular}{m{2.5cm}m{5.5cm}}
\toprule
User & Benefit\\
\midrule
\multirow{3}{*} {Utilities} & $\bullet$ Improved responsiveness of the supply.  \\
&$\bullet$ Eliminate the usage of peaking power plants.\\
&$\bullet$ Improved operations of transmission and distribution systems.\\\hline
\multirow{2} {*}{End-users} & $\bullet$ Reduced electricity costs.     \\
&$\bullet$ Reduced financial losses due to outages.\\\hline\vspace{.05 in}
Independent System Operators\multirow{2}{*} & $\bullet$ Load balancing among regions. \\
& $\bullet$ Stabilization of transmission systems. \\
\bottomrule
\end{tabular}
\end{table}

In this paper, we develop an analytical framework for optimal energy storage sizing. The proposed framework contributes to the existing literature in two ways:
\begin{itemize}
\item The existing analytical methods for storage sizing focus on settings with {\em one} customer. The proposed framework can cope with any network with arbitrary number of consumers. Our analytical results show that a community-level design of storage units exhibits substantial gains over user-level design.
\item The existing methods for multi-consumer settings are simulation-based. The advantage of the proposed analytical method is that it establishes the exact optimal sizing, and subsequently, are computationally less expensive.
\end{itemize}
In the proposed framework consumers' demands are modeled as Markovian fluid and the analysis rely on  stochastic theory of fluid dynamics. We establish the interplay among minimum amount of storage size, the grid capacity, number of consumers, and the stochastic guarantees on outage events.

We note that studying storage units in a network level is of paramount significance as they are expected to become integral to smart energy grids. More specifically, ESS will be employed at smart residential and business complexes and university campuses, to name a few, in order to reduce peak hour consumption. Clearly, in such sharing-based applications, the size of the energy storage is linked to the customer population and the load profile. This relation is far from being linear due to multiplexing gains which is computed by the percentage of reduction in the required amount of resources with respect to baseline case of assigning peak demand to each user.


\begin{table}[t]
  \centering
  \caption{Notations}
    \begin{tabular}{p{0.99cm}  p{6.5cm} }
    \toprule
    Parameter & Description \\
    \midrule
    $C$     & Power drawn from grid. \\
    $N$     & Number of users. Note that this is not the number of houses since in one house there can be multiple appliances requesting demand. \\
    $R_p$     & Demand of a user, the same for all users.\\
    $\lambda$ & Arrival rate of charge request, parameter for Poisson process. \\
    $\mu$ & Mean service rate for the customer demand.\\
    $B$  & Size of the energy storage unit. In the normalized model measured in $R_p$$\mu^{-1}$.\\
    $S(t)$  & ESS depletion level, $0$$\leq$$S(t)$$\leq$$B$. \\
    $L_{i}(t)$ & Aggregated load on the system when $i$ users are ``On??.\\
     ${F_i}(x)$& Steady state cumulative probability distribution function of ESS charge level. \\
    $\varsigma$&Grid power allocated per source ($C$/$N$).\\
    $\kappa$ & ESS per user ($B$/$N$).\\
    \bottomrule
    \end{tabular}
  \label{summary}%

\end{table}%

 \section{System Description}
Consider a community of consumers in which the demands of $i\in\left\{ {1,2,...,N} \right\}$ users are accommodated by the power grid in conjunction with a {\em shared} energy storage system unit of size $B$. We consider a dynamic model for grid capacity, in which capacity fluctuates over time, and the capacity at time $t$ is denoted by $C_t$. for $t\in\mathds{R}^+$. As established in~\cite{sgc13,rosenberg,onOffModels2, Richardson} the consumption pattern of each consumer can be well-represented by a two-state ``On/Off'' process. We define the binary variable $s^i_t$ to represent the state of consumer $i$ at time $t$ such that
\begin{equation}
s^i_t=\left\{
\begin{array}{ll}
1 & \mbox{consumer $i$ is On}\\
0 & \mbox{consumer $i$ is Off}
\end{array}\right. \ .
\end{equation}
When a customer is in the ``On'' state, it initiates an energy demand, where the duration of demand is modeled statistically, which is adopted to capture the variety types of consumers' demands. Specifically, the duration of each customer's demand is assumed to be exponentially distributed with parameter $\mu$. Furthermore we also assume that the requests, which are transitions from ``Off'' to ``On'', are generated randomly and according to a Poisson process with parameter $\lambda$.  Hence,
for each consumer $i$ at any time $t$ we have
\begin{equation}
\mathds{P}(s^i_t\;=\;1)\;=\; \frac{\lambda}{\lambda+\mu}\ .
\end{equation}
Finally, we define $R_p$ as the energy demand per time unit\footnote{We remark that generalization to the settings in which there are multiple customer classes with different demand rates is straightforward and is omitted due to space limitations.}.

In order to formalize the dynamics of ESS, we define $L_i(t)$ and $S(t)$ as the charge request of consumer $i$, and the storage level at the storage unit, respectively. Hence, for the rate of change in the storage level of the ESS, the following holds

{\small
\begin{equation}
\frac{{d{S}(t)}}{{dt}} = \left\{ {\begin{array}{ll}
0&\mbox{if}\;{S}(t)=B\;\mbox{\&}\;\sum_{i=1}^N L_i(t)<C_t\vspace{.1 in}\\
0&\mbox{if}\;{S}(t)=0\;\mbox{\&}\;\sum_{i=1}^N L_i(t) >C_t\vspace{.1 in}\\
{ {{C_t-{\sum\nolimits_i L}_i}(t) }}&\mbox{otherwise}
\end{array}} \right. \ .
\end{equation}
}
Due to stochasticities involved (in consumption and generation), by choosing any storage capacity $B$, only stochastic guarantees can be provided for the reliability of the system and always there exists a chance of outage, which occurs when available resources fall below the aggregate demands by the consumers. By noting that $S(t)$ denotes the energy level that the storage unit needs to feed into the grid to avoid outage, an outage even occurs when the necessary load from the storage unit exceed the maximum available $B$. Hence, we define $\varepsilon$-outage storage capacity, denoted by $B(\epsilon)$, as the smallest choice of $B$ corresponding to which the probability of outage does not exceed $\epsilon\in(0,1)$, i.e.,
\begin{equation}\label{eq:problem}
B(\varepsilon)=\left\{
\begin{array}{ll}
\min & B\vspace{.1 in}\\
{\rm s.t.} & \mathds{P}\big(S_t>B\big)\leq \varepsilon
\end{array}\right. \ .
\end{equation}

Our goal is to determine the $\epsilon$-outage storage capacity $B_\epsilon$ based on grid capacity $C_t$, number of users $N$, and their associate consumption dynamics. The notations are summarized in Table \ref{summary}.


\section{Storage Capacity Analysis}

\subsection{Storage Access Dynamics}
When grid can serve all the consumers' demands there will be no consumer served by the storage unit. On the other hand, when the grid capacity falls below the aggregate demand, the consumers access the storage unit. Since the requests of the consumers arrive randomly, the number of consumers accessing the unit will also vary randomly.

Since we have $N$ {\em independent} consumers each with a two-state model, by taking into account their underlying arrival and consumption processes, the composite model counting the number of users accessing the storage unit at a given time can be modeled as a {\em continuous}-time birth-death process. Specifically, this process consists of $(N+1)$ states, in which state $j\in\{0,\dots,N\}$ models $j$ consumers being active and accessing the storage unit, i.e.,
\begin{equation}
\mbox{state at time $t$ is $j$ if}\quad \sum_{i=1}^Ns^i_t=j\ ,
\end{equation}

and drawing $jR_p$ units of power from the storage unit. As depicted in Fig.~\ref{MMPP} the transition rate from state $j$ to state $j+1$ is $(N-j)\lambda$ and, conversely, the transition from state $j+1$ to state $j$ is $(j+1)\mu$. Hence, for the  associated infinitesimal generator matrix $M$, in which the row elements sum to zero, for $i,j\in\{1,\dots,N+1\}$ we have
{\small
\begin{equation}\label{matrix}
M[i,j]=\left\{
\begin{array}{ll}
-(N-i+1)\lambda -(i-1)\mu & j=i \vspace{.1 in}\\
(i-1)\mu & j=i-1\;\;\&\;\; i\geq 2\vspace{.1 in}\\
(N-i+1)\lambda & j=i+1\;\;\&\;\; i\leq N\vspace{.1 in}\\
0 & \mbox{otherwise}
\end{array}\right. \ .
\end{equation}
}
By denoting the stationary probabilities of state $j\in\{0,\dots,N\}$ by $\pi_j$ and according defining $\bm{\pi}  = \left[ {{\pi _0},{\pi _1},...,{\pi _N}} \right]$, these stationary probability values satisfy $\bm{\pi} M=\bzero$.

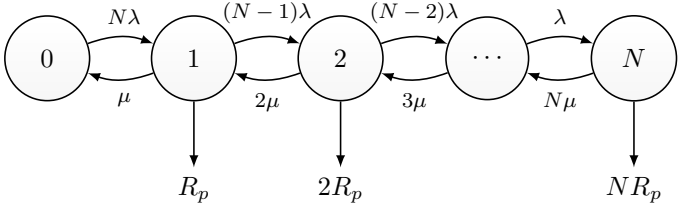
\begin{figure}[t]
\centering
\begin{tikzpicture}[-latex ,auto ,node distance =1.7 cm and 1.95cm ,on grid ,
semithick ,
state/.style ={ circle ,top color =white , bottom color = processblue!20 ,
draw,black , text=black , align=center,text width =0.8 cm, minimum size=12pt}]
\node[state] (C){$0$};
\node[state]       (v1)  [right =of C]   {$1$};
\node[state]       (v2)  [right =of v1]   {$2$};
\node[state]       (v3)  [right =of v2]   {$\cdots$};
\node[state]       (v4)  [right =of v3]   {$N$};
\node[draw=none,below= 1.75cm of v1]  (v1x) {${R_p} $};
\node[draw=none,below= 1.75cm of v2]  (v2x) {${2R_p} $};
\node[draw=none,below= 1.75cm of v4]  (v4x) {${NR_p} $};
\path (C) edge [bend left =20] node[above =0.05 cm, pos=0.5]  {\begin{footnotesize} $N\lambda$\end{footnotesize}  }  (v1);
\path (v1) edge [bend right = -20] node[below =0.05 cm, pos=0.5] {\begin{footnotesize} $\mu$\end{footnotesize}  }(C);
\path (v1) edge [bend left =20] node[above =0.05 cm, pos=0.5] {\begin{footnotesize}$(N-1)\lambda$\end{footnotesize}}  (v2);
\path (v2) edge [bend right = -20] node[below =0.05 cm, pos=0.5] {\begin{footnotesize}$2\mu$\end{footnotesize}}  (v1);
\path (v2) edge [bend left =20] node[above =0.05 cm, pos=0.5] {\begin{footnotesize}$(N-2)\lambda$\end{footnotesize}} (v3);
\path (v3) edge [bend right = -20] node[below =0.05 cm, pos=0.5] {\begin{footnotesize}$3\mu$\end{footnotesize}} (v2);
\path (v3) edge [bend left =20] node[above =0.05 cm, pos=0.5] {\begin{footnotesize}$\lambda$\end{footnotesize}} (v4);
\path (v4) edge [bend right = -20] node[below =0.05 cm, pos=0.5] {\begin{footnotesize}$N\mu$\end{footnotesize}} (v3);
\path    (v1) edge       node[left = 0.05 cm] {} (v1x);
\path    (v2) edge       node[left = 0.05 cm] {} (v2x);
\path    (v4) edge       node[left = 0.05 cm] {} (v4x);

\end{tikzpicture}

\caption{Composite model for $N$ independent users. Each user becomes active (``On'') at rate $\lambda$ and becomes inactive (``Off'') at rate $\mu$. The aggregate demand depends on the active number of users. }\label{MMPP}
\end{figure}

\subsection{Analyzing Distributions}

Given the dynamics of accessing the storage unit, in the next step we analyze the statistical behavior of the ESS charge level. Specifically, we define ${F_i}(t,x)$ as the cumulative distribution function (cdf) of the ESS charge level when $i\in\{0,\dots,N\}$ consumers are depleting the storage unit at time $t$, i.e.,
\begin{equation}\label{eq:cdf}
{F_i}(t,x)\;=\; \mathds{P}\Big (S(t) \leq x \;\;\;\mbox{and}\;\;\; \sum_{j=1}^Ns^j_t=i \Big)\ .
\end{equation}
Accordingly, we define the vector of cdfs as
\begin{equation}
\boldsymbol{F}(t,x)\triangleq\left[ {{F_0}(t,x)\; , \; {F_1}(t,x)\; , \;... \; , \;  {F_N}(t,x)} \right]\ .
\end{equation}
Based on this definition, the next lemma delineates a differential equation which admits the cdf vector as its solution and is instrumental to analyzing the probability of outage events, i.e.,  $\mathds{P}\big(\sum_{i=}^NL_i(t)\;>\; C_t+B\big)$.
\begin{lem}
The cdf vector ${\boldsymbol F}(t,x)$ satisfies
\begin{equation}\label{diffEq}
\frac{d{\boldsymbol F}(t,x)}{dx}\cdot D = {\boldsymbol F}(t,x)\cdot M\ ,
\end{equation}
where $D$ is a diagonal matrix defined as
\begin{equation}\label{D}
D\triangleq \mbox{\rm diag}\left[ { - C_t\mu\; , \; (1 - C_t)\mu\; , \; ... \; , (N - C_t)\mu } \right]\ ,
\end{equation}
and matrix $M$ is defined in \eqref{matrix}.
\end{lem}
\begin{proof}

In order to compute the probability density functions, we find the expansion of $F_i(t,x)$ for an incremental change $\Delta t$ in $t$, i.e, ${F_i}(t + \Delta t,x)$. Note that during incremental time $\Delta t$, three elementary events can occur
\begin{enumerate}
\item one inactive consumer might become active, i.e., $i$ increases to $i+1$;
\item one active consumer might become inactive,  i.e., $i$ reduces to $i-1$; or
\item the number of active consumers remains unchanged.
\end{enumerate}
Since the periods of arrival and departure of consumers are exponentially distributed, corresponding to these events, cdf $F_i(t,x)$ can  be expanded to
\begin{align}\label{generating}
\nonumber {F_i} & (t + \Delta t,x) \\ \nonumber
&= \underbrace{\left[ {N - (i - 1)} \right]\cdot  (\lambda \Delta t)\cdot {F_{i - 1}}(t,x)}_{\text{one consumer added}}\\
\nonumber & + \underbrace{[i + 1]\cdot(\mu \Delta t)\cdot {F_{i + 1}}(t,x)}_{\text{one consumer removed}} \\
\nonumber &+ \underbrace{\left[ {1 - \left( {(N - i)\lambda  + i\mu } \right)\Delta t} \right]\cdot {F_i}(t,x - (i - C_t)\cdot \mu \Delta t)}_{\text{no change}}\\
&  + o\left( {\Delta t^2} \right)\ ,
\end{align}
where $ o\left( {\Delta t^2} \right)$ represents the probabilities of the compound events and tends to zero more rapidly than ${\Delta t^2}$ (and $\Delta t$) as $\Delta t \to 0$. Next, by passing the limit
$$\mathop {\lim }\limits_{\Delta t \to 0} \frac{{{F_i}(t + \Delta t,x)}}{{\Delta t}}$$
it can be readily verified that (\ref{generating}) simplifies to
\begin{align}\label{derives}
\nonumber  \frac{{\partial {F_i}(x,t)}}{{\partial t}} & = \left[ {N - (i - 1)} \right]\cdot  (\lambda)\cdot {F_{i - 1}}(t,x)\\
\nonumber & + [i + 1]\cdot(\mu)\cdot {F_{i + 1}}(t,x)\\
\nonumber & - \left[ {(N - i)\lambda  + i\mu } \right]\cdot {F_i}(t,x)\\
 & - (i - C_t)\cdot(\mu)\cdot \frac{{\partial F_i(t,x)}}{{\partial x}}\ ,
\end{align}
where we have defined ${F_{ - 1}}(t,x)= {F_{N + 1}}(t,x) =0$. Recall that, the main objective is to compute the ESS size that will operate over a long time period. Therefore, it is further assumed that steady state condition holds, that is ${{\partial {F_i}(x,t)} \mathord{\left/
 {\vphantom {{\partial {F_i}(x,t)} {\partial t = 0}}} \right.
 \kern-\nulldelimiterspace} {\partial t = 0}}$.

Hence, (\ref{derives}) can be rewritten as

\begin{align}\label{recursion}
\nonumber (i - C_t)\cdot(\mu)\cdot \frac{{\partial F_i(t,x)}}{{\partial x}} & = \left[ {N - (i - 1)} \right]\cdot  (\lambda)\cdot {F_{i - 1}}(t,x)\\ \nonumber
\nonumber & + [i + 1]\cdot(\mu)\cdot {F_{i + 1}}(t,x)\\
 & - \left[ {(N - i)\lambda  + i\mu } \right]\cdot {F_i}(t,x)\ .
\end{align}
By concatenating all the equations \eqref{recursion} for all $i\in\{0,\dots,N\}$ we obtain the compact form

\begin{equation}\label{diffEq}
\frac{d{\boldsymbol F}(t,x)}{dx}\cdot D = {\boldsymbol F}(t,x)\cdot M\ .
\end{equation}
\end{proof}
The solution of the first order differential equation given in \eqref{diffEq} can be expressed as a sum of exponential terms. The general solution requires computing $(N+1)$ eigenvalues of the matrix $MD^{-1}$ and the general solution is expressed as~\cite{anick1}:
\begin{equation}\label{diff2}
\boldsymbol{F}(t,x) = \sum_{i = 0}^N {{\alpha}_i}\;{\bphi_i}\;\exp(z_ix)\ ,
\end{equation}
where $z_i$ is the $i^{th}$ eigenvalue of $MD^{-1}$ with the associated eigenvector ${\bphi _i}$ which satisfy ${z_i}{\bphi _i}D = {\bphi _i}M$. The coefficients $\{\alpha_0,\dots,\alpha_N\}$ are determined by the boundary conditions, e.g., $F_i(t,0)=0$ and $F_i(t,\infty)=1$.

In order to compute the probability distribution in \eqref{diff2}, we need to determine the eigenvalues of $MD^{{-1}}$, the eigenvectors $\bphi_i$, and coefficients $\alpha_i$.
We notice that, since $x \geq 0$ and $F_{j}(x)$ is upper bounded by $1$, all of the positive eigenvalues and the corresponding $\alpha_i$ must be set to zero, hence this greatly reduces required computational effort and \eqref{diff2} simplifies to
\begin{equation}\label{diff4}
\boldsymbol{F}(t,x) = \sum_{i:Re[z_{i}\leq0]} {{\alpha}_i}\;{\bphi_i}\;\exp(z_ix)\ ,
\end{equation}

We further notice that from ${z_i}{\bphi _i}D = {\bphi _i}M$, one of the eigenvalues must be zero. Then by setting $z_{0}=0$, the corresponding eigenvector can be computed from ${\bphi _0} M=\bzero$. But, recall from the previous discussion that the steady state probability distribution $\bm{\pi}$ of the $N+1$ state Markov chain can also be computed from the same equation, that is $\bm{\pi} M=\bzero$. Since, the eigenvector ${\bphi _0}$ is known and one of the eigenvalues is $z_{0}=0$, we can write ${\bphi _0}=\bm{\pi}$. Therefore, \eqref{diff4} further simplifies to \cite{schwartz1996}
\begin{equation}\label{diff5}
\boldsymbol{F}(t,x) = \boldsymbol{\pi} +\sum_{i:Re[z_{i}< 0]} {{\alpha}_i}\;{\bphi_i}\;\exp(z_ix)\ ,
\end{equation}


\subsection{Single User Storage Capacity ($N=1$)}

In order to establish how to compute the desired $\varepsilon$-outage storage capacity $B(\varepsilon)$ by leveraging the cdf vector found in \eqref{diff5} we start by a simple network with a single user ($N=1$). The insights gained can be leveraged to generalize the approach for networks with any arbitrary size $N$. When $N=1$ the infinitesimal generator matrix $M$ defined in \eqref{matrix} is
\begin{equation}
M\;=\;\begin{bmatrix}
  - \lambda     & \lambda \\
\mu & -\mu
\end{bmatrix}\ .
\end{equation}
For finding the expansion of $\boldsymbol{F}(t,x)$ as given in \eqref{diff2} we need to find $z_0$ and $z_1$ as the eigenvalues of $MD^{-1}$, where $D$ is defined in \eqref{D}. Based on \eqref{D} we find that
\begin{equation}
MD^{-1}\;=\;
\begin{bmatrix}
  \frac{1}{C_t}\cdot \frac{\lambda}{\mu}     & -\frac{1}{1-C_t}\cdot \frac{\lambda}{\mu} \\
-\frac{1}{C_t} & -\frac{1}{1-C_t}
\end{bmatrix}\ .
\end{equation}
Hence, the eigenvalues are
\begin{equation}\label{eigenValue}
z_0=0\quad\mbox{and}\quad z_1=\frac{\chi}{C_t}-\frac{1}{{1 - C_t}}
\end{equation}
where we have defined $\chi\triangleq \frac{\lambda}{\mu}$. It can be readily verified that the eigenvector associated with $z_1$ is $\bphi_1=[1-C_t\;,\; C_t]$. Therefore, according to \eqref{diff5} we have
\begin{align}\label{diff6}
\boldsymbol{F}(t,x) = \boldsymbol{\pi} +{\alpha}_1\;{\bphi_1}\;\exp(z_1x)\ .
\end{align}
Finally, by finding the coefficient $\alpha_1$ we can fully characterize $\boldsymbol{F}(t,x)$. This can be facilitated by leveraging the boundary condition $F_1(t,0)=0$, which yields
\begin{equation}
F_1(t,0)\;= \; \pi_1 \; + \; \alpha_1\;C_t\;=\; 0 \ ,
\end{equation}
where we have that $\pi_1=\frac{\lambda}{\lambda+\mu}$. Therefore
\begin{equation}\label{alphaValue}
\alpha_1\;=\; -\frac{\chi}{C_t(1+\chi)}\ ,
\end{equation}
which subsequently fully characterizes both cdfs $F_0(t,x)$ and $F_1(t,x)$ according to
\begin{eqnarray*}
F_{0}(t,x) & = & \pi_{0}+\alpha_1(1-C_t)\exp(z_1x)\\
\mbox{and}\quad F_{1}(t,x) & = & \pi_{1}+\alpha_1 C_t\exp(z_1x)\ .
\end{eqnarray*}
As a result, by recalling the definition of $F_i(t,x)$ in \eqref{eq:cdf}, the probability that the storage level $S_t$ falls below a target level $x$ is given by
\begin{equation}\label{mainResult}
\mathds{P}(S_t \leq x)\; = \;  {F_0}(x) + {F_1}(x) = 1 + \alpha_1 \exp(z_1x)\ .
\end{equation}
Given this closed-form characterization for the distribution of $S_t$, we can now evaluate the probability term
\begin{equation}\label{eq:P}
\mathds{P}(S_t>B)\ ,
\end{equation}
which is the core constraint in the storage sizing problem formalized in \eqref{eq:problem}. Specifically, for any instantaneous realization of $C_t$  denoted by $c$ we have
\begin{align}\label{eq:St}
\nonumber \mathds{P}(S_t>B) & = \int_{C_t}\mathds{P}(S_t>B\;|\; C_t=c)\;f_{C_t}(c)\;dc\\
\nonumber & =-\int_{C_t}\alpha_1\exp(z_1B)\;f_{C_t}(c)\;dc\\
\nonumber & =\int_{C_t}\frac{\chi}{c(1+\chi)}\exp\left( \frac{B\chi}{c}-\frac{B}{1 -c}\right)f_{C_t}(c)\;dc\ .
\end{align}
Therefore, by noting that $z_1=\frac{\chi}{c}-\frac{1}{1-c}$ is negative, the probability term $\mathds{P}(S_t>B)$ becomes strictly decreasing in $B$. Hence, the smallest storage capacity $B$ that satisfies the stochastic guarantee $\mathds{P}(S_t>B)\leq \varepsilon$ has a unique solution corresponding to which this constraint holds with equality. In the simplest settings in which grid capacity $C_t$ is constant $c$ we find
\begin{align}
B(\varepsilon) &= \frac{c(1-c)}{\chi - \chi c-c} \cdot \log\frac{\varepsilon c(1+\chi)}{\chi}\ .
\end{align}
\subsection{Multiuser Storage Capacity ($N>1$)}

In this subsection we provide a closed form for the probability term $\mathds{P}(S_t \leq x)$ for arbitrary values of $N$, which we denote by $F_{N}(x)$. Computing all $F_N(x)$ through computing its constituent terms $F_i(t,x)$, especially as $N$ grows, becomes computationally expensive, and possibly prohibitive as  it involves computing the eigenvalues and eigenvectors of $MD^{{-1}}$. By capitalizing on the observation that for large number of users $N\gg 1$, the largest eigenvalues are the main contributors to the probability distribution~\cite{morrison1989} shows that, the asymptotic expression for $F_N(x)$ is given by
\begin{align}\label{Nusers}
F_{N}(x) & = \frac{1}{2}\sqrt {\frac{u}{{\pi f(\varsigma )(\varsigma  + \lambda (1 - \varsigma ))N}}}\\\nonumber
& \;\;\times \exp(- N\varphi (\varsigma)- g(\varsigma )x)\\\nonumber
& \;\;\times \exp(- 2\sqrt {\left\{ {f(\varsigma )(\varsigma  + \lambda (1 - \varsigma ))Nx} \right\}})\\[-3mm]\nonumber
& \;\;\nonumber
\end{align}
where,
{\small
\begin{align*}
f(\varsigma ) & \triangleq \log \left( {\frac{\varsigma }{{\lambda (1 - \varsigma )}}} \right) - 2\frac{{\varsigma (1 + \lambda ) - \lambda }}{{\varsigma  + \lambda (1 - \varsigma )}\ ,}\\
u &  \triangleq  \frac{{\varsigma (1 + \lambda ) - \lambda }}{{\varsigma (1 - \lambda )}}\ ,\\
\varphi (\varsigma ) & \triangleq \varsigma \log (\varsigma ) + (1 - \varsigma )\log (1 - \varsigma ) - \varsigma \log (\varsigma ) + \log (1 + \lambda )\ ,\\
g(\varsigma ) & \triangleq k + 0.5\left( {\varsigma  + \lambda (1 - \varsigma )} \right)\frac{{\psi (1 - \varsigma )}}{{f(\varsigma )}}\ ,\\
k & \triangleq (1 - \lambda ) + \frac{{\lambda (1 - 2\varsigma )}}{{(\varsigma  + \lambda (1 - \varsigma ))}}\ ,\\
\psi  & \triangleq \frac{{(2\varsigma  - 1){{(\varsigma (1 + \lambda ) - \lambda )}^3}}}{{\varsigma {{(1 - \varsigma )}^2}{{(\varsigma  + \lambda (1 - \varsigma ))}^3}}}\ .
\end{align*}
}
In this set of equations, time is measured  in units of a single average ``On'' time ($1/\mu$). Furthermore, $\kappa$ and $\varsigma$ are defined as the ESS per user ($B/N$) and the grid power allocated per source, respectively. Furthermore, we denote the variable $\upsilon $ as the power above the mean demand allocated per user as $\upsilon  = \varsigma  - \frac{\lambda }{{1 + \lambda }}$.

\section{Numerical Examples}\label{examples}
\begin{figure}[t]
\centering
\includegraphics[width=\columnwidth]{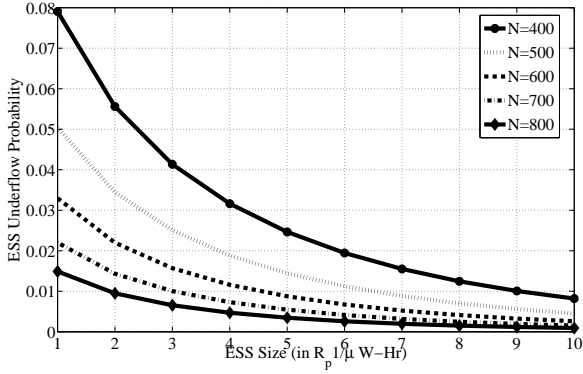}
 \caption{Community storage sizing for different user population.}\label{NvsESS}
\end{figure}
\begin{figure}[t]
\centering
\includegraphics[width=\columnwidth]{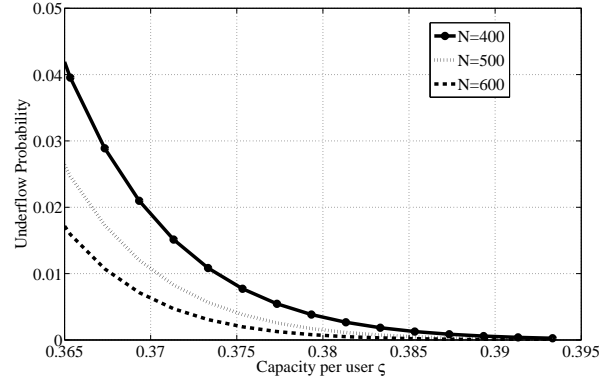}
 \caption{Community storage sizing for varying grid power. }\label{capacityOverFlow}
\end{figure}
\begin{figure}[t]
\centering
\includegraphics[width=\columnwidth]{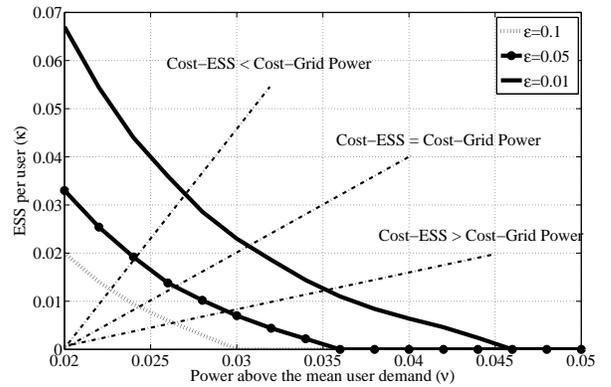}
 \caption{Power allocation (per user) above the mean user demand.}\label{BufferVsPower}
\end{figure}
\begin{figure}[t]
\centering
\includegraphics[width=\columnwidth]{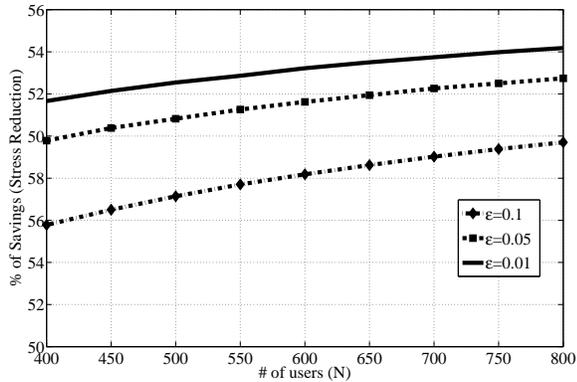}
 \caption{Percentage of savings in grid power.}\label{savings}

\end{figure}

In this section, we provide several numerical examples to explain the system dynamics and show how the proposed framework can be used in typical peak shaving applications. We use the aforementioned normalized values (unit time is measured in $\mu^{-1}$ and unit demand is measured in peak demand - $R_p$). We start by exploring the relations between the number of users, ESS size (in $R_p \mu^{-1}$ units) and the corresponding underflow probability for a given system capacity $C$.  Charge request rate per single user $\lambda$ is set to $2$ (two charge requests of size $R_p$ arrives in unit time), and the mean capacity above the mean demand per user is set to $\upsilon=0.035$. Then the total system capacity becomes $C=0.3683N$ units. In Fig.~\ref{NvsESS}, ESS sizing is evaluated for user population ($N$) from $400$ to $800$. This result can be used in various ways. First, given user population $N$, system operator can choose the ESS size according to a certain underflow probability. For instance, for a large scale EV charging facility (e.g., located in shopping mall, airports~\cite{sgc13}) with $N$=$400$ charging slots in order to accommodate $99\%$ of the customer demand the ESS size should be selected as $B=9\times10(=R_p)\times0.5 (=\mu^{-1})=45$ kWh. One important thing to notice is that as the user population increases the required ESS size reduces due to the increase in multiplexing gains. Another important observation is that instead of sizing the ESS to meet the entire customer demand, just by rejecting a few percentage of customers, great savings in the storage size, hence in terms of total system cost, can be achieved.

Another design consideration from the system operator point of view would be the following. Suppose that the system operator employs an already acquired ESS size of $B$=$5$, then she is interested in the amount of power to draw from the grid so that she can guarantee to meet certain level of demand (e.g., $99$\% etc.). To that end, the underflow probability for a range of system capacity per user $\varsigma$ and user population is evaluated in Fig.~\ref{capacityOverFlow}. Obviously as the capacity per user increases the underflow probability goes to zero. Similar to previous evaluation, as the number of  users increase, due to multiplexing gains the same percentage of customers can be accommodated with less amount of resources.

Next the relation between the ESS capacity and the grid power is investigated for a fixed number of users $N=500$. Each curve depicted in Fig.~\ref{BufferVsPower} represents a contour of underflow probability and the Buffer-Grid Power ($B$-$C$) combinations to reach the same underflow probabilities. Obviously in order to serve more customers (less $\epsilon$) more grid and ESS capacity are required. Moreover this result can be useful from financial analysis standpoint. For a specific project system (number of customers $N$, underflow probability) designer can analyze the unit cost of ESS and grid resources. Then the optimal combination of the ESS-Grid power can be obtained at the intersection of the cost curve and the contours given here. Three different cases for the cost are illustrated in Fig.~\ref{BufferVsPower}. As a future work, we are aiming to develop cost models for energy storage and power grid to optimally compute grid and ESS resources.

The primary motivation for the employment of the ESS is to reduce the stress on the grid and improve the utilization of power system components (e.g., power generation etc.). Thus, our final evaluation is on the percentage of reduction on the power grid for a fixed ESS size ($B$=$5$) and for different underflow probabilities. This time arrival rate is set to $\lambda$=$4$ and the comparison for varying arrival rates are done according to peak demand allocation. It can be seen in Fig.~\ref{savings} that multiplexing resources lead to great reduction on the power grid. In a similar manner, in Fig. \ref{savingsUser}, we compute the percentage of savings in ESS size with respect to $N=10$ users.

\begin{figure}[!t]
\centering
\includegraphics[width=\columnwidth]{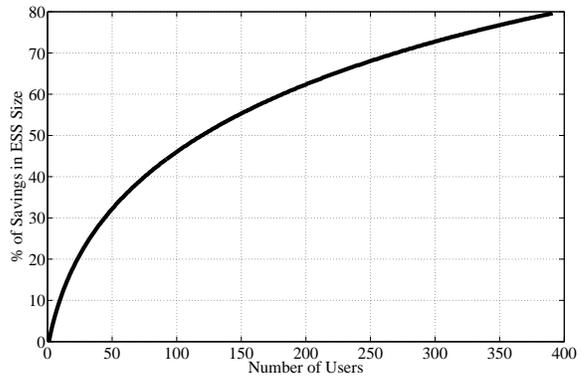}
 \caption{Percentage of savings in ESS size with respect to $N=10$ users.}\label{savingsUser}
 \vspace{-13pt}
\end{figure}
 \section{Conclusion}
In this paper we provided an analytical framework to size a sharing-based energy storage system for peak hour utility applications such as load leveling, peak shaving, and energy arbitrage. The analysis establishes the interplay among dynamic grid capacity, the number of consumers, and the guarantees levels for avoiding outage events. The analysis and simulation results exhibit substantial gains of community-level storage sharing compared to the settings in which the consumers have their dedicated storage units.

\ifCLASSOPTIONcaptionsoff
  \newpage
\fi

\bibliographystyle{IEEEtran}
\bibliography{energycon}
\vfill
\end{document}